\documentclass[11pt]{amsart}
\usepackage{amsmath,amsthm,amsfonts,amssymb,amscd}
\usepackage[dvips]{graphicx}
\usepackage{psfrag}
\usepackage{a4wide}
\usepackage{xcolor}

\theoremstyle{plain}
\newtheorem{thm}{Theorem}[section]

\newtheorem{lemma}[thm]{Lemma}

\newtheorem{maintheorem}{Theorem}

\theoremstyle{definition}

\newtheorem{defi}[thm]{Definition}

\newcommand{\RR}{{\mathbb R}}

\renewcommand{\epsilon}{\varepsilon}

\newcommand{\diam}{\operatorname{diam}}

\newcommand{\cC}{\EuScript{C}}

\newcommand{\cH}{\EuScript{H}}

\newcommand{\cO}{\EuScript{O}}

\newcommand{\sing}{\operatorname{Sing}}

\def \RR {{\mathbb R}}

\def \cC {{\mathcal C}}

\def \cH {{\mathcal H}}

\def \cO {{\mathcal O}}

\def \Hom {{\mathcal  Hom}}
\def \SQ {{\mathcal  SQ}}

\begin{document}

\title[$N$-expansive flows]
{ $N$-expansive flows}

\author{
 Alfonso Artigue, Welington Cordeiro
and Maria Jos\'e Pac\'ifico 
 }
 \thanks{M.J.P. was partially supported by CNPq,
 FAPERJ. \,
 A.A. was partially supported by PEDECIBA, ANII}

\maketitle

\begin{abstract}{We define the concept of $N$-expansivity for flows and extend some of the results already established for discrete dynamics and for  $CW$-expansive flows. 
We show examples of $N$-expansive flows but not expansive, and examples of $CW$-expansive flows but not $N$-expansive for any natural number $N$.}
We also define Komuro $N$-expansivity and prove that on compact surfaces it implies Komuro expansivity.
\end{abstract}


\section{Introduction}\label{sec:intro} 

In this paper we consider the dynamics of flows on compact metric spaces
with some mild kind of expansiveness.
The notion of expansiveness was introduced in the middle of the twentieth century by Utz \cite{Ut} and, in the 90s, Kato introduced the notion of \emph{continuum-wise expansiveness, cw-expansiveness for short,} for homeomorphisms \cite{Ka}.
Later,  in \cite{ACP}, this notion was extended to flows and the authors
proved, among many other fundamental properties, that the topological entropy of a cw-expansive flow defined on  compact metric spaces with topological dimension greater than 1 is positive. The concept of  $N$-expansive flows was introduced
 by Cordeiro in his Thesis \cite{C}, extending  the notion of $N$-expansivity for homeomorphisms given in \cite{M}. In the context of discrete dynamics, it was proved in \cite{APV} that if $f$ is a $2$-expansive homeomorphism defined on a compact boundaryless surface $M$ with non-wandering set $\Omega(f)$ being the whole of $M$ then $f$ is expansive. It was also shown that this condition on the non-wandering set can not be relaxed.
In \cite{CC}  is analyzed the dynamics of $N$-expansive homeomorphisms satisfying the shadowing property. 
On compact surfaces, in \cite{A} some examples of Axiom A
diffeomorphisms are given which are $N$-expansive but not $(N-1)$-expansive. Moreover, it is shown that they are robustly $N$-expansive in the $C^r$ topology, for $r=N$.

The goal of this work is to give  a  definition of \emph{$N$-expansiveness} for flows,  extend some of the results already established for discrete dynamics for $N$-expansive flows and
prove in this context some of the results in \cite{ACP}.

  To announce precisely our result, let us recall some concepts and definitions already established. We  point out that throughout this paper, $M$ denotes a compact metric space.
 
A \emph{flow} in $M$ is a family of homemorphisms $\{X^t\}_{t\in\mathbb{R}}$ satisfying $X^0(x)=x$, for all $x\in M$ and $X^{t+s}(x)=X^t(X^s(x))$ for all $s,\, t \in \RR$ and $x\in M$. A \emph{continuum} is a compact connected set and it is \emph{non-degenerate} if it contains more than one point. We denote by $\cC(M)$ the set of all continuum subsets of $M$.

Let $\Hom_{0}(\mathbb{R})$ be the set of homeomorphisms on $\RR$
fixing the origin and if $A$ is a subset of $M$, $C^0(A,\RR)$ denotes
the set of real continuos maps defined on $A$. Given $A\subset M$, we denote  $C^0(A)$ the set of real continuous maps on $A$ and define $\cH(A)$ by {
\begin{equation}\label{def-H(A)} \{\alpha:A\rightarrow \Hom_{0}(\mathbb{R}); \, \exists\,\, x_\alpha\in A \, \mbox{with} \  \alpha(x_\alpha)=id_\mathbb{R} \ \mbox{and}  \ \alpha(.)(t)\in C^0(A), \forall \,\, t\in\mathbb{R} \}.
\end{equation}}
If $t\in\mathbb{R}$ and $\alpha\in \cH(A)$, set $\mathcal{X}^t_\alpha(A)=\{X^{\alpha(x)(t)}(x); x\in A\}$. 
See Figure 1.

\begin{figure}[htb]
\begin{center}
\includegraphics[height=4cm]{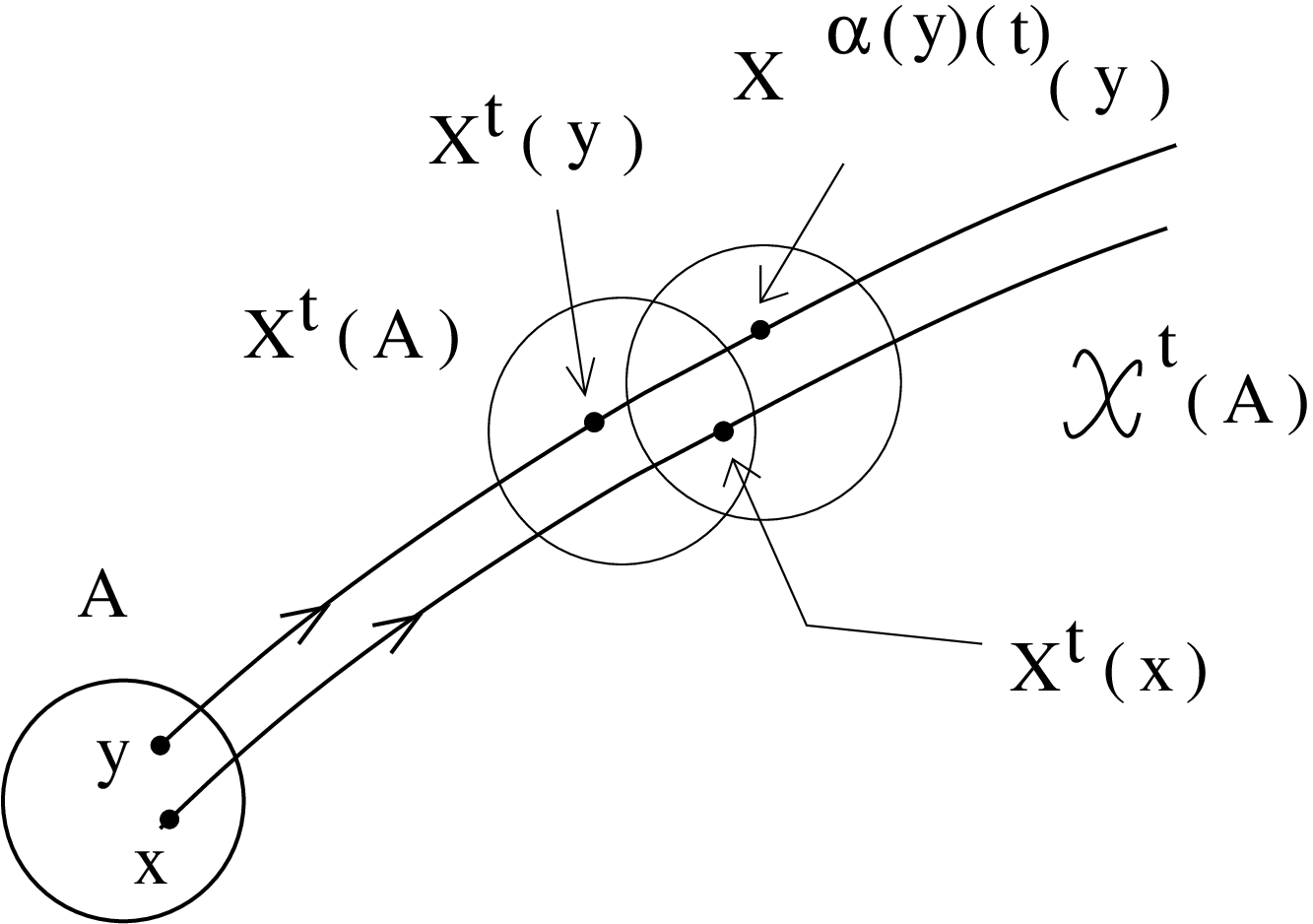}
\caption{ The set  $\mathcal{X}^t_\alpha(A)$} \label{fig1}
\end{center}
\end{figure}

\begin{defi}\label{def-n}
If $N>0$, we say a flow $X^t$ is \emph{$N$-expansive} if for any $\epsilon>0$ there is a $\delta>0$ such that if $A\subset M$ and $\alpha\in \cH(A)$ satisfies 
$$\diam(\mathcal{X}^t_\alpha(A) )<\delta, \ \ \forall t\in\mathbb{R},$$
then there is a subset $B\subset A$, with cardinality less than $N+1$ such that $A\subset X^{(-\epsilon,\epsilon)}(B)$.
\end{defi}

\begin{defi}\label{def-cw}\cite{ACP}
A flow $X^t$ is \emph{$cw$-expansive} if for any $\epsilon>0$ there is a $\delta>0$ such that if $A\subset M$ is a continuum and $\alpha\in \cH(A)$ satisfies 
$$\diam(\mathcal{X}^t_\alpha(A) )<\delta, \ \ \forall t\in\mathbb{R},$$
then $A\subset X^{(-\epsilon,\epsilon)}(x_\alpha)$.
\end{defi}


Our first results show that in the context of flows,  $N$-expansiveness implies cw-expansiveness 
and that  $N$-expansiveness is a conjugacy invariant. 

\begin{maintheorem}\label{t-Nimplica-cw} If a flow is $N$-expansive, with $N>0$, then it is $cw$-expansive. 
\end{maintheorem}

Therefore, by \cite{ACP}, if a $N$-expansive flow is defined on compact metric space with dimension greater than $1$, then the topological entropy of the flow is positive.  

\begin{maintheorem}\label{A}
$N$-expansiveness is a conjugacy invariant.
\end{maintheorem}

The next result establishes that the converse to Theorem \ref{t-Nimplica-cw} fails.

\begin{maintheorem}\label{teoA} There are examples of $cw$-expansive flows defined on a compact metric space, but not $N$-expansive for any natural number $N$. Moreover, for each $N \geq 2$ there are examples of flows $X^t$ and $Y^t$ defined on compact metric spaces $M$ and $L$ respectively, satisfying:
\begin{enumerate}
\item $X^t$ is $N$-expansive, but not expansive;
\item $Y^t$ is $N+1$-expansive, but not $N$-expansive.
\end{enumerate} 
\end{maintheorem}

The  result below shows an equivalent notion of $N$-expansivity using a
fixed collection of cross-sections with special properties, called $\delta$-adequate
(see Definition \ref{adequada}).

\begin{maintheorem} \label{teoB} A flow $X^t$ is $N$-expansive if, and only if, given a $\delta$-adequate  pair $(\mathcal{S},\mathcal{T})$ of cross-sections
there is  $\eta>0$ such that  
$$\#W_\eta^s(x)\cap W_\eta^u(x)\leq n, \ \forall x\in\bigcup_{i=1}^kT_i. $$  
\end{maintheorem}

As we will see in Lemma \ref{fixedpoint}, a $N$-expansive flow of a connected compact metric space (not reduced to a singleton) cannot have singular points.
Thus, from \cite[Remark 3.1]{A3} (and references therein)
we conclude that there is no $N$-expansive flows on a compact surface.
Motivated by this fact, we introduce a definition of Komuro $N$-expansivity, generalizing Komuro expansivity \cite{Ko}.

\begin{defi}
If $N>0$, we say a flow $X^t$ is \emph{Komuro $N$-expansive} if for any $\epsilon>0$ there is a $\delta>0$ such that if $A\subset M$ and $\alpha\in \cH(A)$ satisfies
$$\diam(\mathcal{X}^t_\alpha(A) )<\delta, \ \ \forall t\in\mathbb{R},$$
then there are $t_0\in\mathbb{R}$ and a subset $B\subset A$, with cardinality at most $N$ such that $\mathcal{X}^{t_0}_\alpha(A)\subset X^{(-\epsilon,\epsilon)}(B)$.
\end{defi}

Notice that Komuro 1-expansive is equivalent to Komuro expansive.
We introduce another definition related to H. Kato's \emph{continuum-wise expansivity}.

\begin{defi}
If $N>0$, we say a flow $X^t$ is \emph{Komuro cw-expansive} if for any $\epsilon>0$ there is a $\delta>0$ such that if $A\subset M$ is connected
and $\alpha\in \cH(A)$ satisfies
$$\diam(\mathcal{X}^t_\alpha(A) )<\delta, \ \ \forall t\in\mathbb{R},$$
then there are $t_0\in\mathbb{R}$ and $x\in A$ such that $\mathcal{X}^{t_0}_\alpha(A)\subset X^{(-\epsilon,\epsilon)}(x)$.
\end{defi}

It was proved in \cite{APV} that expansivity and N-expansivity are not equivalent in the context of homeomorphisms defined on compact surfaces. In contrast, for flows on compact surfaces,  Komuro expansivivity is equivalent to Komuro N-expansivity (for every N), as it is shown in the next result.

\begin{maintheorem}\label{D}
A flow on a compact surface is
Komuro $N$-expansive if and only if it is
Komuro expansive.
There are Komuro cw-expansive flows on compact surfaces which are not Komuro $N$-expansive.
\end{maintheorem}

To finish, let us introduce a definition in order to indicate a generalization of Theorem \ref{D} and state an open problem or future direction of research.

\begin{defi}
A flow $X^t$ is \emph{Komuro finite(countable)-expansive} if for any $\epsilon>0$ there is a $\delta>0$ such that if $A\subset M$ and $\alpha\in \cH(A)$ satisfies
$$\diam(\mathcal{X}^t_\alpha(A) )<\delta, \ \ \forall t\in\mathbb{R},$$
then there are $t_0\in\mathbb{R}$ and a finite (countable) subset $B\subset A$ such that $\mathcal{X}^{t_0}_\alpha(A)\subset X^{(-\epsilon,\epsilon)}(B)$. Recall that $ \cH(A)$
is as at (\ref{def-H(A)}).
\end{defi}

It seems clear that the proof of Theorem \ref{D} works for Komuro finite-expansive flows on surfaces.
To get the result to countable-expansivity, it is needed first to control or bound the number of singularities or to work directly with infinitely many  singularities.
It would be interesting to characterize Komuro countable-expansive flows of compact surfaces.

This paper is organized as follows: in Section \ref{sec-basic} we prove basic properties satisfied by $CW$-expansive flows and prove Theorems  \ref{t-Nimplica-cw} and \ref{A}; in Section \ref{section2} we prove the relation between $N$-expansivity of homemomorphisms and their suspension flows, and we apply these properties to prove Theorem \ref{teoA}; in Section \ref{section3} we use Keynes and Sears's  techniques to find good properties in cross sections of $N$-expansive flows and prove Theorem \ref{teoB};
in Section \ref{secSingNexp} we prove Theorem \ref{D}.

Acknowledgement. We thank the anonymous referee for calling our attention to the reference \cite{LMS} where it is considered the notion of $N$-expansivity for flows from a previous version of the present paper.

\section{Basic properties of $N$-expansive Flows}\label{sec-basic}

Let $M$ be a compact metric space and $X^t$ a $cw$-expansive flow on $M$.   
In this section we prove basic results satisfied by $X^t$.
Given $x \in M$, let $\cO(x)=\{ X^t(x), \, t \in \RR\}$ be the orbit of $x$. A point $x\in M$ is  \emph{fixed} by $X^t$ if $\cO(x)=\{ x\}$.  
Recall that $\cO(x)$ is regular if it is not a unique point. 

\begin{lemma} If $X^t$ is a $N$-expansive flow, then each fixed point of $X^t$ cannot be acumulated by fixed points. 
\end{lemma}
\begin{proof} The proof is simple and goes by contradiction. Let $p$ be a fixed point of $X^t$ accumulated by fixed points and $\delta>0$ be given by the $N$-expansive property for $\epsilon=1$. Therefore there are infinitely many points in $B_\delta(p)$. Let $A$ be a set with only $N+1$ of these fixed points and $p\in A$. Hence for any $\alpha\in \mathcal{H}(A)$ we have
$$\diam\mathcal{X}^t_\alpha(A)=\diam A<\delta, \ \forall t\in\mathbb{R},$$   
leading to a contradiction, because since the flow is $N$-expansive, $A$ should have at most $N$ different orbits.
\end{proof}

Since each $N$-expansive flow is a $cw$-expansive flow, by \cite[Lemma 2.1]{ACP} each fixed point cannot be accumulated by regular points. Moreover, each $N$-expansive flow  has only a finite number of fixed points. Combining this with the previous lemma we have,

\begin{lemma}\label{fixedpoint} If $X^t$ is a $N$-expansive flow, then there are a finite number of fixed points and each one is an isolated point of the space.
\end{lemma} 

Let
$\Sigma_\mathbb{R}(0)=\{\{x_i\}_{i\in\mathbb{Z}}; x_i\in\mathbb{R} \ \mbox{and} \ x_0=0\}.
$
If $A\subset M$ define
\begin{eqnarray*} \SQ(A)=\{\beta:A\rightarrow\Sigma_\mathbb{R}(0); \exists \,\, x_\beta \ \mbox{such that} \ \beta(x_\beta)_i\rightarrow\infty \ \mbox{when} \ i\rightarrow\infty \ \mbox{and} \\ \beta(x_\beta)_i\rightarrow-\infty \ \mbox{when} \ i\rightarrow-\infty\}.
\end{eqnarray*}
Let $\SQ^*(A)\subset \SQ(A)$ so that if $\beta\in \SQ^*(A)$, then for each  $i\in\mathbb{Z}$ the map 
\begin{eqnarray*} f_i:A\rightarrow\mathbb{R}, \quad \quad
f_i(a)=\beta(a)_i \quad \mbox{is continuous}.
\end{eqnarray*}
If $\beta\in \SQ^*(A)$ and $i\in\mathbb{Z}$ define $\mathcal{X}_\beta^i(A)=\{X^{\beta(x)(i)}(x); \ x\in A\}$.

\begin{thm} Let $X^t$ be a flow without fixed points. The following properties are equivalent:
\begin{enumerate}
\item[(1)] $X^t$ is $N$-expansive;
\item [(2)] $\forall\,\,\eta>0$ there is $\delta>0$ so that if there is $\alpha\in \cH(A)$ with $\diam(\mathcal{X}^t_\alpha(A))<\delta$ $\forall t\in\mathbb{R}$, then $A$ is contained in at most $N$ orbit segments inside $B_\eta(x_\alpha)$;
\item[(3)] $\forall\epsilon>0$ there is $\delta>0$ such that if there exists $\beta\in SQ^*(A)$, with $\beta(x_\beta)_{i+1}-\beta(x_\beta)_{i}\leq \eta$ and $\sup_{a\in A}|\beta(a)_{i+1}-\beta(a)_{i}|\leq \delta$ $\forall i\in\mathbb{Z}$, such that $\diam(\mathcal{X}^i_\beta(A))<\delta$ $\forall i\in\mathbb{Z}$, then there is a subset $B\subset A$, with cardinality less than $n+1$ such that $A\subset X^{(-\epsilon,\epsilon)}(B)$.
\end{enumerate}
\end{thm}
\begin{proof} $(1) \Rightarrow (3):$ Take any $\epsilon>0$ and let $\delta>0$ be given by $N$-expansivity property. Choose $\eta>0$ such that 
$$\eta+(2\sup_{z\in M, |u|<\eta}d(z,X^u(z)))<\frac{\delta}{2}. $$
Assume that for some $\beta\in \SQ^*(A)$, with $\beta(x_\beta)_{i+1}-\beta(x_\beta)_{i}\leq \eta$ and $\sup_{a\in A}|\beta(a)_{i+1}-\beta(a)_{i}|\leq \delta$  $\forall i\in\mathbb{Z}$, we have $\diam(\mathcal{X}^i_\beta(A))<\eta$, $\forall i\in\mathbb{Z}$.  For each $a\in A$ we define a homeomorphism $\alpha(a)$ by $\alpha(a)(\beta(x_\beta)_i)=\beta(a)_i$, $\forall i\in\mathbb{Z}$ and by linearity on each interval $(\beta(x_\beta)_i,\beta(x_\beta)_{i+1})$. Then $\alpha\in \cH(A)$. 
Therefore, if $t_i=\beta(x_\beta)_i$, for each $t\in[t_i,t_{i+1})$, we get
\begin{eqnarray*} \diam(\mathcal{X}_\beta^t(A))&=&\sup_{a,b\in A} d(X^{\beta(a)(t)}(a),X^{\beta(b)(t)}(b)) \\
&\leq&\sup_{a,b\in A} (d(X^{\beta(a)(t)}(a),X^t(x_\beta))+d(X^t(x_\beta),X^{\beta(b)(t)}(b))) \\
&\leq&\sup_{a\in A}d(X^{\beta(a)(t)}(a),X^t(x_\beta))+\sup_{b\in A}d(X^t(x_\beta),X^{\beta(b)(t)}(b))) \\
&=&2\sup_{a\in A}d(X^t(x_\beta), X^{\beta(a)(t)}(a))\,.
\end{eqnarray*} 

But, for each $a\in A$ we have that
\begin{eqnarray*} d(X^t(x_\beta), X^{\beta(a)(t)}(a))&\leq&d(X^t(x_\beta),X^{t_i}(x_\beta))+d(X^{t_i}(x_\beta),X^{\beta(a)_i}(a)) \\ &+& d(X^{\beta(a)_i}(a),X^{\beta(a)(t)}(a)) \\
&\leq& \sup_{z\in M, |u|\leq\eta}d(z,X^u(z))+\eta+\sup_{z\in M, |u|\leq\eta}(z,X^u(z)) \\
&<& \frac{\delta}{2}\,.
\end{eqnarray*}
Therefore,
\begin{eqnarray*} \diam(\mathcal{X}_\beta^t(A))< 2\frac{\delta}{2}=\delta,\quad \forall t\, \in\mathbb{R},
\end{eqnarray*}
and since $X^t$ is a $N$-expansive flow there exists $B\subset A$ with at most $N$ points such that $A\subset X^{(-\epsilon,\epsilon)}(B)$.
\vspace{0.2cm}

$(3)\Rightarrow (1)$: Let $\epsilon>0$ be given and $\delta>0$ as in item $(3)$. Suppose $A\subset M$ is compact such that for some $\alpha\in \cH(A)$ it holds $\diam(\mathcal{X}^t_\alpha(A))<\delta$ $\forall t \in\mathbb{R}$. By induction, define $\beta\in \SQ^*(A)$ by $\beta(x)_0=0$ $\forall x\in A$. Since $A$ is compact there is $t_1>0$ such that $\alpha(x)(t_1)<\delta$ $\forall x\in A$. Define $\beta(x)_1=\alpha(x)(t_1)$. There is $t_{i+1}>t_i$ such that $|\alpha(x)(t_{i+1})-\alpha(x)(t_i)|<\delta$ $\forall x\in A$, define then $\beta(x)_{i+1}=\alpha(x)(t_{i+1})$. We can define $\beta(x)_i$ for $i<0$ similarly. Therefore $\beta\in \cH(A)$, with $x_\beta=x_\alpha$. By item $(3)$  there exists $B\subset A$ with at most $N$ points such that $A\subset X^{(-\epsilon,\epsilon)}(B)$. \vspace{0.2cm}    

$(1)\Rightarrow (2)$: Since $M$ is compact, $\forall \eta>0$ there is $\epsilon>0$ such that $X^{(-\epsilon,\epsilon)}(x)\subset B_\eta(x)$ $\forall x\in M$.\vspace{0.2cm}

$(2)\Rightarrow(1)$:  Since $X^t$  has no fixed points, $\forall\epsilon>0$ there is $\eta>0$ so that $\diam(X^{(-\epsilon,\epsilon)}(x))>\eta$ $\forall x\in M$. 
\end{proof}

Next we prove Theorem \ref{t-Nimplica-cw}.

\begin{proof}[Proof of Theorem \ref{t-Nimplica-cw}] Assume that the flow $X^t$ is not $cw$-expansive, i.e., there is $\epsilon_0>0$ such that for each $\delta>0$ there is a continuum $A_{\delta}$ and $\alpha\in \cH(A_\delta)$ with
$$\diam(\mathcal{X}^t_\alpha(A_\delta) )<\delta, \ \ \forall t\in\mathbb{R},$$
but $A_\delta$ is not a subset of $X^{(-\epsilon_0,\epsilon_0)}(x_\alpha)$.
If $X^t$ is $N$-expansive, then we can take $\delta>0$ from the definition of $N$-expansive flow for $$\epsilon=\frac{\epsilon_0}{2N}.$$
Let $A=A_\delta$ be as above
and take $B\subset A$, with cardinality less than $N+1$ such that $A\subset X^{(-\epsilon,\epsilon)}(B)$.
From Lemma \ref{fixedpoint} we can assume that there are not fixed points.
Therefore, taking a smaller $\delta$ if needed, we can suppose that
$A$ is contained in a flow box.
This and the fact that $B$ is finite implies that $A$ is a continuum piece of orbit.
On the one hand, since $A$ is a continuum and $x_\alpha\in A$, we have that there is no continuum piece of orbit with time interval $(0,\epsilon_0)$ containing $A$. On the other hand, a union of $N$ pieces of orbits with time intervals $$(-\epsilon,\epsilon)=(\frac{-\epsilon_0}{2N},\frac{\epsilon_0}{2N})$$ cannot form an interval with time interval longer than $(0,\epsilon_0)$. This contradiction proves the result.
\end{proof}

Recall that if $M_1$ and $M_2$ are metric spaces,
the flows $X^t: M_1 \to M_1$ and $Y^t:M_2 \to M_2$ are \emph{conjugate} if there is a homeomorphism  $h:M_1\to M_2$ mapping orbits of $X^t$ onto orbits of ~$Y^t$.
Next we prove Theorem \ref{A}, establishing  that $N$-expansivity is a conjugacy invariant.

\begin{proof}[Proof of Theorem \ref{A}]
Suppose $X^t$ and $Y^t$ are conjugate with $Y^t$ $N$-expansive. Let $h:M_1 \rightarrow M_2$ be a homeomorphism conjugating $X^t$ to $Y^t$. Let $\epsilon_1>0$ be given and $\epsilon_2>0$ such that if $x,y\in M_2$ satisfies $d(x,y)<\epsilon_2$, then $d(h^{-1}(x),h^{-1}(y))<\epsilon_1$. Let $\delta_2>0$ be the corresponding number given by $N$-expansivity of $Y^t$ to $\epsilon_2$ and   $\delta_1>0$ such that if $x,y\in M_1$ with $d(x,y)<\delta_1$, then $d(h(x),h(y))<\delta_2$. Therefore, if $A_1\subset M_1$ is a compact set and $\alpha_1\in \cH(A_1)$ is such that $\diam(\mathcal{X}^t_{\alpha_1}(A_1))<\delta_1$ $\forall t\in\mathbb{R},$ we get that $A_2=h(A_1)\subset M_2$ is a compact set. 
For every $x\in A_1$ and $t\in\mathbb{R}$ let $\alpha_2(h(x))(t)$ be the real number such that if
$$h(X^{\alpha_1(x)(t)}(x))=Y^{\alpha_2(h(x))(t)}(h(x))$$
then $\alpha_2 \in \cH(A_2)$. 
Furthermore, for every $t\in\mathbb{R}$ we have:
\begin{eqnarray*} \diam(\mathcal{Y}^t_{\alpha_2}(A_2))&=&\max_{x,y\in A_2}d(Y^{\alpha_2(x)(t)}(x),Y^{\alpha_2(y)(t)}(y)) \\
&=& \max_{x,y\in A_1}d(h(X^{\alpha_1(x)(t))}(x)),h(X^{\alpha_1(y)(t))}(y))) \\
&=& \diam(h(\mathcal{X}^t_{\alpha_1}(A_1))) <\delta_2,\\
\end{eqnarray*}
because  $\diam(\mathcal{X}^t_{\alpha_1}(A_1))<\delta_1$. 
Since $Y^t$ is $N$-expansive, $A_1$ is contained in at most $N \,\, Y^t$-orbit segments inside $B_{\epsilon_2}(x_{\alpha_2})$. But $A_1=h^{-1}(A_2)$, $x_{\alpha_2}=h(x_{\alpha_1})$ and the choice of $\epsilon_2$ imply that $A_1$ is contained in at most $N \,\, X^t$-orbit segments inside $B_{\epsilon_1}(x_{\alpha_1})$, proving that $X^t$ is $N$-expansive.   This finishes the proof.
\end{proof}

\section{Proof of Theorems \ref{t-Nimplica-cw} and \ref{teoA}}\label{section2}

Let $(M,d)$ be a compact metric space and
 $f:M\rightarrow M$ a homeomorphism. Let $k:M\rightarrow\mathbb{R}^+$ be a continuous map. 

\begin{defi}The \emph{suspension} of $f$ under $k$ is the flow $X^t$ on the space 
$$M_k=\bigcup_{0\leq t\leq k(y)}
(y,t)/(y,k(y))\sim(f(y),0)$$
defined for small nonnegative time by $X^t(y,s)=(t+s)$, $0\leq t+s< k(y)$. 
\end{defi}
Each suspension of $f$ is conjugate to the suspension of $f$ under $1$, the constant function with value $1$. For this reason we shall concentrate on suspensions under the constant map  $k(x)=1,\,\,\forall x\in M$.

Next, following \cite{BW}, we define a metric on $M_1$. Suppose that the diameter of $M$ under $d$ is less than $1$. 
Consider the subset $M\times\{t\}$ of $M\times[0,1]$ and let $d_t$ denote the metric defined by 
$$d_t((y,t),(z,t))=(1-t)d(y,t)+td(f(y),f(z)), \quad y,z\in M.$$ 
Given $x_1,x_2\in M_1$, consider all finite chains $x_1=w_0,w_1,...,w_n=x_2$ between $x_1$ and $x_2$ where, for each $i$, either $w_i$ and $w_{i+1}$ belong to $M\times{t}$ for some $t$ (in wich case we call $[w_i,w_{i+1}]$ a horizontal segment) or $w_i$ and $w_{i+1}$ are on the same orbit (and then we call $[w_i,w_{i+1}]$ a vertical segment). \\

\begin{figure}[htb]
\begin{center}
\includegraphics[height=3.5cm]{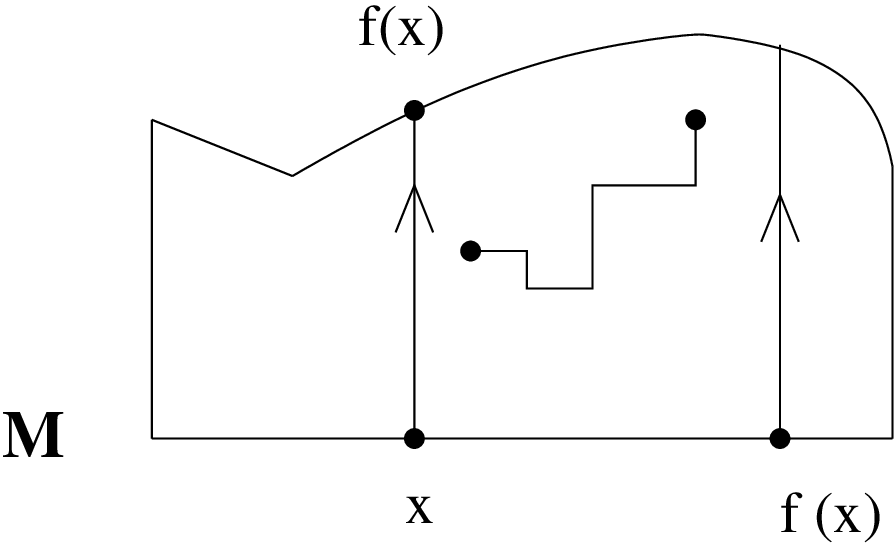}
\caption{Length of a segment in the metric $d_t$} \label{figura2}
\end{center}
\end{figure}

Define the length of a chain as the sum of the lengths of its segments, where the length of a horizontal segment $[w_i,w_{i+1}]$ is measured in the metric $d_t$ if $w_i$ and $w_{i+1}$ belongs to $M\times\{t\}$, and the length of a vertical segment $[w_i,w_{i+1}]$ is the shortest distance between $w_i$ and $w_{i+1}$ along the orbit  using the usual metric on $\mathbb{R}$. See Figure~\ref{figura2}.

If $w_i\neq w_{i+1}$ and $w_i$ and $w_{i+1}$ are on the same orbit and on the same set $M\times\{t\}$ then the length of the segment $[w_i,w_{i+1}]$ is taken as $d_t(w_i,w_{i+1})$, since this is always less than $1$. 

Then define $d(x_1,x_2)$ to be the infimum of the lengths of all chains between $x_1$ and $x_2$. It is easy to see that $d$ is a metric on $M_1$.
This metric $d$ gives the topology on $M_1$.

\begin{thm}\label{th-sus} Let $\phi:M\rightarrow M$ be a homeomorphism and $f:M\rightarrow\mathbb{R}^+$ a continuous map. The suspension of $\phi$ under $f$ is a $N$-expansive flow if and only if $\phi$ is a $N$-expansive homeomorphism.
\end{thm}

\begin{proof} We need only show the result when $f\equiv 1$.
Suppose that $X^t$ is $N$-expansive. Let $\frac{1}{2}>\epsilon>0$ be given and $\delta>0$ be the corresponding constant determined by $N$-expansivity property. If $A\subset M$ is a subset with $\diam(\phi^n(A))<\delta$ $\forall n\in\mathbb{Z}$, denoting  $A_M=A\times\{0\}$, and $x_1=(x,0)$, then for all $t\in\mathbb{R}$ we have:
\begin{eqnarray*} \diam(X^t(A_M))&=&\max_{x_1,y_1\in A_M}d(X^t(x_1),X^t(y_1)) \\
&\leq& \max_{x,y\in A}\rho_{t-[t]}((\phi^{[t]}(x),t-{[t]}),(\phi^{[t]}(y),t-{[t]})) \\
&=& \max_{x,y\in A}((1-t+[t])\rho(\phi^{[t]}(x),\phi^{[t]}(y))+(t-[t])\rho(\phi^{[t]+1}(x),\phi^{[t]+1}(y))).
\end{eqnarray*}  
Where $[t]$ is the greatest integer less than $t$. But since for all $x,y\in A$ it holds
$$\phi^{[t]}(x),\phi^{[t]}(y)\in\phi^{[t]}(A) \ \mbox{and} \ \phi^{[t]+1}(x),\phi^{[t]}(y)\in\phi^{[t]+1}(A)\, ,$$
we get
\begin{eqnarray*} \diam(X^t(A_M))&\leq&(1-t+[t])\delta+(t-[t])\delta=\delta. 
\end{eqnarray*}
Since $X^t$ is $N$-expansive, there is $B_M\subset A_M$ with at most $N$ points such that $A_M\subset X^{(-\epsilon,\epsilon)}(B_M)$. Moreover since $0<\epsilon<\frac{1}{2}$ and $A_M\subset M\times\{0\}$ we obtain $A_M=B_M$, and so $A=\{x; (x,0)\in A_M\}$. Therefore, $\phi$ is $N$-expansive.

Next suppose that the homeomorphism $\phi$ is $N$-expansive. Consider in $M$ the metric given by
$$\rho'(x,y)=\min\{\rho(x,y),\rho(\phi(x),\phi(y))\}$$
and let $\delta>0$ be the $N$-expansivity constant to $\rho'$. Let $\epsilon>0$ and $\delta'=\min\{2\delta,\epsilon,\frac{1}{4}\}$. 
Let $A\subset M_f$ and $\alpha\in \cH(A)$ be such that $\diam\mathcal{X}^t_\alpha(A)<\delta'$ $\forall t\in\mathbb{R}$. 
We consider two cases: (1)  Each point $x\in A$ can be represented as $(y_x,1/2)$ and (2) when (1) not happens.

In the first case, define $A_M=\{a\in M;(a,1/2)\in A \}$. Then
\begin{eqnarray*}\diam(A_M)&=&\max_{y,z\in A_M}\rho'(y,z) \\
&=& \max_{(y,0),(z,0)\in A}d((y,0),(z,0)) \\
&=&\diam(A)<\delta'<\delta.
\end{eqnarray*}  
By definition of suspension, for each $x\in A$ we have that $X^{1}(x)$ has representation $(\phi(y_x),1/2)$, and since $\diam\mathcal{X}^1_\alpha(A)<\delta<\frac{1}{4}$ we get that $X^{\alpha(y)(1)}(y)$ has representation $(\phi(y),s)$ with $s\in(0,1)$ $\forall y\in A$. So,
\begin{eqnarray*}\diam(\phi(A_M))&=&\max_{y,z\in A_M}\rho'(\phi(y),\phi(z)) \\
&\leq& \max_{(\phi(y),s),(\phi(z),r)\in \mathcal{X}_\alpha^1(A)}d((\phi(y),s),(\phi(z),r)) \\
&=&\diam(\mathcal{X}^1_\alpha(A))<\delta'<\delta.
\end{eqnarray*}
Similarly, $\diam(\phi^n(A_M))<\delta$ $\forall n\in\mathbb{Z}$. Since 
$\phi$ is $N$-expansive  $A_M=\{y_1,...,y_j\}$ with $j\leq n$, and so, every point in $A$ has the form $(y_i,0)$ with $1\leq i\leq j$. Finishing the proof in the first case.
\vspace{0.2cm}

For the second case, for each point $x\in A$ there is $r_x$, with $|r_x|<\frac{1}{2}$, such that $X^{r_x}(x)$ has representation as $(y_x,\frac{1}{2})$. Define $\widetilde{A}=\{X^{r_x}(x);x\in A\}$. Define $\Theta$ the family of subsets $B\in A$ such that for each point $y\in\widetilde{A}$ there is exactly one point $x\in B$ such that $y=X^{r_x}(x)$. Since $A$ is connected, there is at least one $B\in\Theta$ such that $B$ is connected.  For each each $x\in B$ and $t\in\mathbb{R}$ set $$\widetilde{\alpha}(X^{r_x}(x))(t)=\alpha(x)(t+r_x)-\alpha(x)(r_x).$$
Then $\alpha\in\mathcal{H}(\widetilde{A})$ and for every $t\in\mathbb{R}$, it holds
$$\diam(\mathcal{X}^t_{\widetilde{\alpha}}(\widetilde{A}))<2\delta.$$
By the first case, we obtain that $\widetilde{A}=\{y_1,...,y_j\}$ with $j\leq n$, therefore $B$ has $j$ points and 
$$A\subset X^{(-\epsilon,\epsilon)}(B).$$

All together completes the proof of Theorem \ref{th-sus}.
\end{proof} 

\begin{proof}[Proof of Theorem \ref{teoA}]
In \cite[Theorem A]{CC} the authors  exhibit, to each $N\geq 2$, a $N$-expansive homeomorphism with the shadowing property, that is not $N-1$-expansive. In \cite[Theorem 5.1]{A} the author exhibits,   for each $N\geq 2$, a $N$-expansive $C^N$-diffeomorphism defined on a surface $S$ that is not $N-1$-expansive. We can apply the above theorems for these examples to find, for each $N\geq 2$, examples of $N$-expansive flows which are not $N-1$-expansive. Recall that for flows and homeomorphisms, $1$-expansivity is equivalent to expansivity. By \cite[Theorem 3.2]{ACP} and \cite[Theorem 6]{BW} we can use the suspensions of the these examples to find for each $N\geq 2$ a $N$-expansive flow that is not a $N-1$-expansive.

In \cite[section 2.2]{A2} the author shows that the \emph{pseudo Anosov with $1$-prong} is a $cw$-expansive homeomorphism, but is not a $N$-expansive homemomorphism for all $N\geq 1$. By theorem \ref{th-sus} and by \cite[Theorem 6]{BW} we get that the suspension of the pseudo Anosov with $1$-prong is $cw$-expansive flow, but is not $N$-expansive for all $N\geq 1$. 
\end{proof}

\section{Proof of  Theorem \ref{teoB}}\label{section3}

In this section our analysis of  a $N$-expansive flow $X^t$ on $M$
will be carried out relative to a fixed collection of cross-sections with special properties. 
We use the notation introduced by Keynes and Sears in \cite{KS} and start recalling the definition of $\delta$-adapted family of cross-sections. \vspace{0.1cm}

A set $S\subset M$ is a \emph{cross-section} of time $\epsilon>0$ if it is closed and for each $x\in S$ we have $S\cap X^{(-\epsilon,\epsilon)}(x)=\{x\}$. The \emph{interior} of $S$ is the set $S^*=int(X^{(-\epsilon,\epsilon)}(S))\cap S$.
The proof of the next lemma can be found in \cite[Lemma 2.4]{KS} as well in \cite[Lemma 7]{BW}.

\begin{lemma}\label{lst} Let $X^t$ be a continuous flow without fixed points. There is $\epsilon>0$ such that for each $\delta>0$ we can find a pair $(\mathcal{S},\mathcal{T})$ of finite families $\mathcal{S}=\{S_1,...,S_n\}$ and $\mathcal{T}=\{T_1,...,T_n\}$ of local cross-sections of time $\epsilon>0$ and diameter at most $\delta$ with $T_i\subset S_i^*$ ($i\in\{1,...,k\}$) such that
$$M=\bigcup_{i=1}^{k} X^{[0,\epsilon]}(T_i)=\bigcup_{i=1}^{k} X^{[-\epsilon,0]}(T_i)=\bigcup_{i=1}^{k} X^{[0,\epsilon]}(S_i)=\bigcup_{i=1}^{k} X^{[-\epsilon,0]}(S_i).$$   
\end{lemma}

By lemma \ref{fixedpoint} we can consider only expansive flows without fixed points.

\begin{defi}\label{adequada}
A pair of families of cross-sections $(\mathcal{S},\mathcal{T})$  as in the previous lemma is called {\emph{$\delta$-adequate}}.
\end{defi}
Given a  pair of   $\delta$-adequate cross sections $(\mathcal{S},\mathcal{T})$ let

\begin{equation}\label{e.theta} 
\theta=\sup\{\delta>0; \forall x\in\bigcup_{i=1}^{k}S_i \ \mbox{it holds} \ X^{(0,\delta)}(x)\cap\bigcup_{i=1}^{k}S_i=\emptyset\}\,.
\end{equation}

Let $\rho>0$ satisfying $5\rho<\epsilon$ and $2\rho<\theta$. And for each $S_i$ consider $D_\rho^i=X^{(-\rho,\rho)}(S_i)$ and define the projection 
\begin{equation}\label{e.projecao} P_\rho^i:D_\rho^i\rightarrow S_i
\end{equation}
by $P_\rho^i(x)=X^t(x)$, were $X^t(x)\in S_i$ for $|t|<\rho$. Let $\frac{1}{2}\theta>\epsilon_0>0$ be such that if $x,y\in S_i$, $d(x,y)<\epsilon_0$ and $t$ is a real number with $|t|<3\delta$ and $X^t(x)\in T_j$, then $X^t(y)\in D^j_\rho$.

Let  $\phi:\bigcup_{i=1}^k T_i\to \bigcup_{i=1}^k T_i$ be  the
{\em{ first return map}} defined as $\phi(x)=X^t(x)$ where $t>0$ is the smallest positive number such that $X^t(x)\in \bigcup_{i=1}^k T_i$. Note that $t\in[\theta,\epsilon]$.

If $x\in T_i$ and $y\in S_i$ with $d(x,y)<\epsilon_0$ let $\{y^x_0,...,y^x_n\}\subset \cO_X(y)$ such that $y^x_0=y$ and $y^x_j=P_\rho^l(X^t(y^x_{j-1}))$, where $t>0$ is the smallest positive time such that $\phi^j(x)=X^t(\phi^{j-1}(x))$, and $l$ is such that $\phi^j(x)\in T_l$. We can continue this construction while $d(\phi^j(x),y_j)<\epsilon_0$. Similarly to $j<0$. See Figure \ref{fig3}.

The stable and unstable sets of points is defined in the following way.
If $x\in T_i$ and $\eta<\epsilon_0$, the \emph{$\eta$-stable set} of $x$ is
is defined as
$$ W^s_\eta(x)=\{y\in S_i; d(\phi^i(x),y_i)<\eta \ \forall i\geq 0\}$$       
and the \emph{$\eta$-unstable set} of $x$ is defined as
$$ W^u_\eta(x)=\{y\in S_i; d(\phi^i(x),y_i)<\eta \ \forall i\leq 0\}\,.$$

\begin{figure}[htb]
\begin{center}
\includegraphics[height=4.5cm]{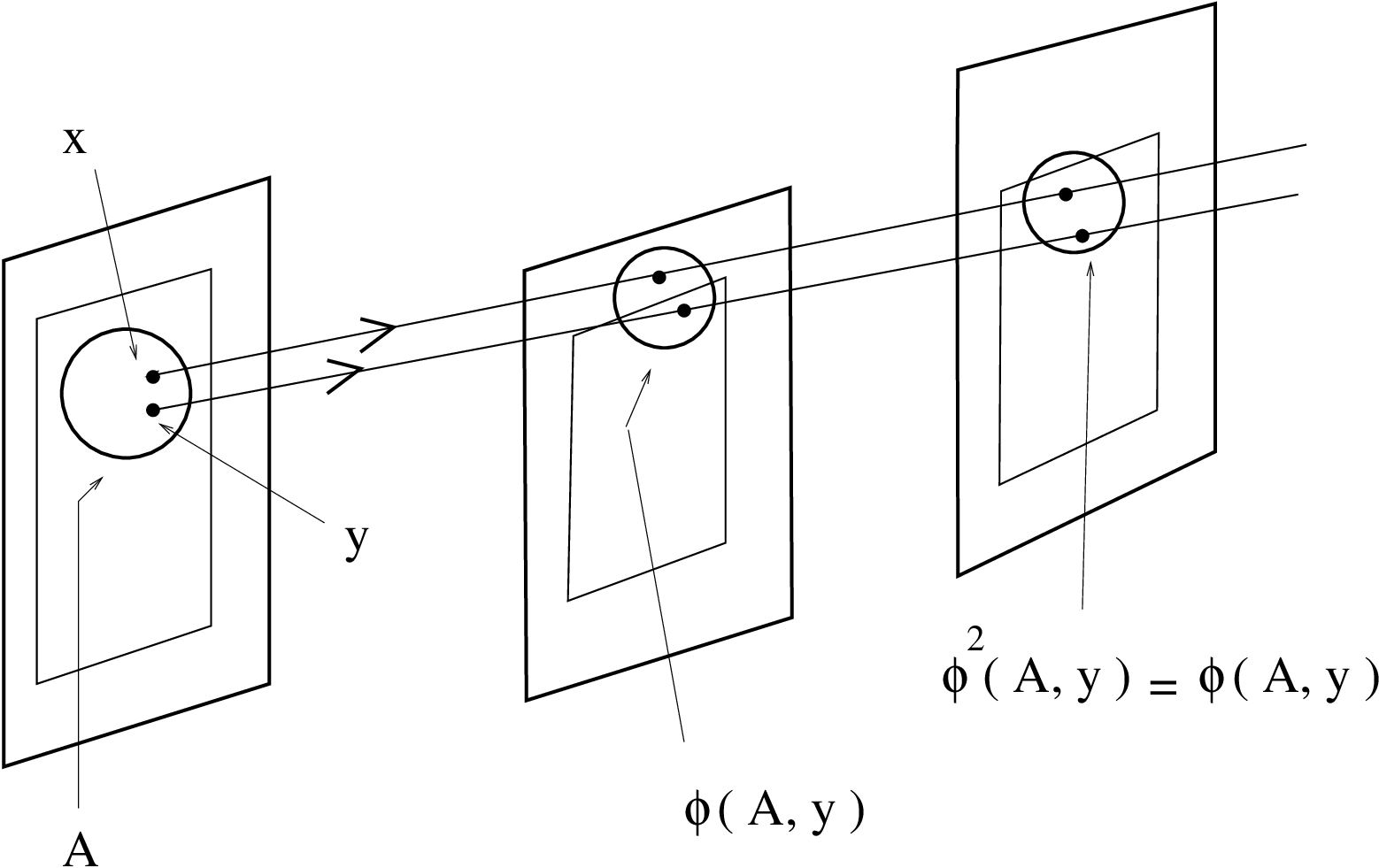}
\caption{The first return map $\phi:\bigcup_{i=1}^k T_i\to \bigcup_{i=1}^k T_i$ } \label{fig3}
\end{center}
\end{figure}

\bigskip
Theorem 2.7 in \cite{KS} establishes that a 
 flow $X^t$ is expansive if, and only if, given a pair $(\mathcal{S},\mathcal{T})$ $\delta$-adequate, there is an $\eta>0$ such that  $W^s_\eta(x)\cap W^u_\eta(x)=\{x\}$ for every $ x\in\bigcup_{i=1}^kT_i$,.  

\begin{proof}[Proof of Theorem \ref{teoB}]
Suppose that $X^t$ is $N$-expansive. Let $a\in(0,\epsilon)$ be given and $a_1>0$ be the $N$-expansive constant corresponding to $a$. Let $\eta\in(0,\epsilon_0)$ be such that if $p\in T_i$ and $q\in S_i$, $i\in\{1,...,k\}$, with $d(p,q)<\eta$, then $d(X^t(p),X^s(q))<\frac{a_1}{2}$ where, if $X^{t_1}(p)=\phi(p)$ and $X^{s_1}(q)=q_1$, then $t\in[0,t_1]$ and $s\in[0,s_1]$ and $|s-t|\leq |s_1-t_1|$. Suppose $y\neq x, y\in W^s_\eta(x)\cap W^u_\eta(x)$ and $\{t_j\}$ and $\{s_j\}$ are the increasing bisequences suh that $X^{t_j}(x)=\phi^j(x)$ and $X^{s_j}(y)=y_j^x$. Define a piecewise linear function $$h:\mathbb{R}\rightarrow\mathbb{R}$$ such that $h(t_j)=s_j$. Now for each $t\in\mathbb{R}$, $t=t_j+u$ and $h(t)=s_j+\widetilde{u}$, where $0\leq u\leq t_{j+1}-t_j$, $0\leq\widetilde{u}\leq s_{j+1}-s_j$ and
$$|u-\widetilde{u}|\leq |(t_{j+1}-t_j)-(s_{j+1}-s_j)|$$
for some $j$, and by the choice of $\eta$, we have 
$$d(X^t(x),X^{h(t)}(y))=d(X^{u}(\phi^j(x)),X^{\widetilde{u}}(y^x_j))<\frac{a_1}{2}.$$
Since every point in $W^s_\eta(x)\cap W^u_\eta(x)$ lie in the same $S_i$, if $y_1\neq y_2$ with $$y_1,y_2 \in W^s_\eta(x)\cap W^u_\eta(x)$$ then $y_1\notin X^{(-a,a)}(y_2)$. Therefore, by the $N$-expansivity of $X^t$, $W^s_\eta(x)\cap W^u_\eta(x)$ has at most $N$ points.

To prove the reverse implication, suppose that given a pair $(\mathcal{T},\mathcal{S})$ and $\rho>0$, there is $\eta>0$ such that 
$$\#W_\eta^s(x)\cap W_\eta^u(x)\leq N.$$

Let $h:\mathbb{R}\rightarrow\mathbb{R}$ be a continuous function with $h(0)=0$. Let $x\in T_i$ and $y\in S_i$, for some $i$, and $\{t_j\}$ and $\{s_j\}$ be such that $X^{t_j}(x)=\phi^j(x)$ and $X^{s_n}(y)=y^x_j$. 
Choose $\delta_0\in(0,\epsilon-\delta-\rho)$ and positive numbers, $a_2<a_1<\eta$, and $a_3,a_4>0$ such that if $u\in T_i$ and $v\in S_i$ then it holds
\begin{enumerate}
\item $d(u,v)<a_1$ implies $d(u,X^t(v))>a_1$ for all $|t|\in[\delta_0,\epsilon]$;
\item $d(u,v)<a_2$ implies $d(\phi(u),v^u_1)<a_1$;
\item $d(u,v)\geq a_2$ implies $d(u,X^t(v))>a_3$ for all $|t|<\delta$;
\item If $x,y\in M$ and $d(x,y)<a_4$ then $d(X^t(x),X^t(y))<a_1$ for all $|t|<\delta(\alpha)$.
\end{enumerate}   

We shall prove that if $\epsilon'=\min(a_2,a_3,a_4)$, $x\in T_i$, $y\in\bigcup_{i=1}^kS_i-W_\eta^s(x)\cap W_\eta^u(x)$, and $h:\mathbb{R}\rightarrow\mathbb{R}$, with $h(0)=0$, then for some $t\in\mathbb{R}$, we have $$d(X^t(x),X^{h(t)}(y))>\epsilon'.$$

\begin{description}
\item[(a)] Suppose that for each $j\in\mathbb{Z}$ we have  
$$|h(t_i)-s_i|<\delta.$$
By hypothesis, there is $j\in\mathbb{Z}$ such that $d(\phi^j(x),y^x_j)>\eta$. Therefore, $$d(X^{t_j}(x),X^{s_j}(y))>\eta>a_1>a_2.$$ By $(3)$ we have $$d(X^{t_j}(x),X^{h(s_j)}(y))>a_3.$$

\item[(b)] Suppose that $j\in\mathbb{Z}$ is the integer with smallest modulus such that   
$$|h(t_i)-s_i|\geq\delta.$$
We can assume $j>0$. 
\item[(b.1)] Suppose there is $i\in [0,j)$ such that
$$ d(\phi^i(x),y_i)>a_2.$$
Reasoning as in (a),  we get
$$d(X^{t_j}(x),X^{h(s_j)}(y))>a_3.$$
\item[(b.2)] Suppose that for all $i\in [0,j)$ we have 
$$ d(\phi^i(x),y_i)\leq a_2<a_1.$$
\item[(b.2.1)] Suppose $t=s_j-h(t_j)\geq\delta_0$. If $h(t_j)\geq s_{j-1}-\delta_0$, then
$$\delta\leq t\leq s_j-s_{j-1}+\delta_0<\delta+\rho+\delta_0<\epsilon$$
and by $(1)$,
$$d(X^{t_j}(x),X^{h(t_j)}(y) = d(X^{t_j}(x),X^{s_j-t}(y))>a_1.$$

But if $h(t_j)<s_{j-1}-\delta_0$, since $h(t_j)>s_{j-1}-\delta_0$ we can find $t'\in (t_{j-1},t_j)$ such that $h(t')=s_{j-1}-\delta_0$. Let $\zeta=t'-t_{j-1}>0$. Then,
$$d(X^{t_{j-1}}(x),X^{s_{j-1}}(y))<a_2<a_1.$$
 $(1)$ implies that
$$d(X^{t_{j-1}}(x),X^{s_{j-1}-\delta_0-\zeta}(y)>a_1.$$  
Therefore, by $(4),$ 
$$ d(X^{t'}(x),X^{h(t')}(y))=d(X^{t_{j-1}+\zeta}(x),X^{s_{j-1}-\delta_0}(y)))>a_4.$$

\item[(b.2.2)] Suppose $t=h(t_j)-s_j\geq\delta_0$. Thus $h(t_j)\geq s_j+\delta_0$ and $h(t_{j-1})\leq s_{j-1}+\delta_0<s_j+\delta_0$ and therefore there is $t'\in (t_{j-1},t_j]$ and $h(t')=s_j+\delta_0$. Let $\zeta=t_j-t'$. Then $d(X^{t_j}(x),X^{s_j}(y))<a_1$ $(1)$ implies $$d(X^{t_j}(x),X^{\delta_0+\zeta}(X^{s_j}(y)))>a_1.$$
Hence 
$$d(X^{t'}(x),X^{h(t')}(y))=d(X^{t_j-\zeta}(x),X^{s_j+\delta_0})>a_4,$$
by $(4)$.
\end{description}

Now assume that $x$ and $y$ are arbitrary points of $M$. Choose $\delta_1>0$ and $a_5>0$ such that $d(x,y)<a_5$ implies
\begin{enumerate}
\item $d(X^t(x),X^s(y))<\epsilon'$, where $t$ is the smallest positive number such that $$X^t(x)\in\bigcup_{i=1}^kT_i, \ X^s(y)=P_\rho(X^t(y)) \ \text{and} \ |t-s|<\frac{1}{16}\delta_1;$$
\item $d(X^w(x),X^v(y))<\epsilon'$, for $|w|,|v|\leq\delta_1+\delta$, and $|v-w|<\delta_1$.
Let $a_6>0$ be such that if $d(x,y)\leq a_6$ then $d(X^{t'}(x),X^{t'+t}(y))>a_6$ for $$\frac{1}{16}\delta_1\leq|t|\leq\zeta \ \text{and} \ |t'|<\delta.$$
Finally, take $a_5>a_8>a_7>0$ such that for each $z\in B_{a_7}(x)$ the set of all points in $B_{a_8}(z)$ which project onto $P_\rho(X^t(z))$ is contained in $X^{(-a_5,a_5)}(z)$, and $B_{a_9}(z)\subset B_{a_8}(x)$.
\end{enumerate}
Assume $d(x,y)<a_7$ and $y\notin X^{(-a_5,a_5)}(z)$, where $P_\rho(z)\in W^s_\eta(x)\cap W^u_\eta(x)$. 
\begin{description}
\item[(a)] Suppose $|h(t)-s|<\frac{1}{8}\delta_1$. Choose $\delta'>\frac{1}{4}\delta_1$ such that $|h(t+t')-s|<\frac{1}{4}\delta_1$, for each $|t'|<\delta'$. Define a homeomorphism $$\widetilde{h}:\mathbb{R}\rightarrow\mathbb{R}$$ by $\widetilde{h}(t)=h(t+t')-s$, for all $|t'|\geq\delta'$, $\widetilde{h}(0)=0$ and otherwise by linearity. Therefore, 
$$|\widetilde(t')-t'|\leq \frac{1}{4}\delta_1+\delta'<\frac{1}{2}\delta_1, \ |t'|<\delta'.$$
Then, by $(1)$ and $(2)$ 
$$d(X^{t+t'}(x),X^{s+\widetilde{h}(t')}(y))<\epsilon', \ |t'|<\delta'.$$
By the first part of the proof we can find $t_0\in\mathbb{R}$ such that $$d(X^{t_0}(X^t(x)),X^{\widetilde{h}(t_0)}(X^s(y)))>\epsilon'$$ and by the construction $|t_0|$ must be greater than $\delta'$, i.e.,
$$d(X^{t+t_0}(x),X^{h(t+t_0)}(y))>\epsilon'.$$
 	  
\item[(b)] If $|h(t)-s|\geq\frac{1}{8}\delta_1$ then $|h(t)-t|\geq\frac{1}{16}\delta_1$ and therefore, if we choose $t'\leq t$ such that $|h(t')-t'|=\frac{1}{16}\delta_1$, hence either
$$d(X^{t'}(x),X^{h(t')}(y))=d(X^{t'}(x),X^{t'\pm\frac{1}{16}\delta_1}(y))>a_6,$$
or $d(x,y)>a_6$.
 
Then $\min\{\epsilon',a_6,a_7\}$ is a $N$-expansive constant corresponding to $a_5$.\qedhere
\end{description}
\end{proof}

\section{Komuro N-expansivity on surfaces}
\label{secSingNexp}

In this section we consider Komuro $N$-expansive flows.
Let $\sing(X)$ denote the set of singular points of the flow.
To do so, we first prove two auxiliary lemmas:
\begin{lemma}\label{l-sing-finito}
If $X$ is Komuro $N$-expansive and $M$ is a compact space then $\sing(X)$ is a finite set.
\end{lemma}
\begin{proof}
The proof goes by contradiction.
Suppose that $\sing(X)$ has infinitely many points.
As $M$ is compact, for any $\delta>0$ we can take a subset
$A\subset \sing(X)$ with $N+1$ different points such that $\diam(A)<\delta$.
Thus, for any $\alpha\in \cH(A)$ we have
$\mathcal{X}^t_\alpha(A)=A$ and
$\diam(\mathcal{X}^t_\alpha(A) )<\delta$ for all $t\in\mathbb{R}$.
But, as $A$ has $N+1$ fixed points, there is no subset $B$ with at most $N$ points
such that
$\mathcal{X}^{t_0}_\alpha(A)\subset X^{(-\epsilon,\epsilon)}(B)$ and $X$ cannot be Komuro $N$-expansive.
This contradiction proves that Komuro $N$-expansive flows on compact spaces
have finitely many fixed points.
\end{proof}

\begin{lemma}\label{l-omega-tudo}
If a flow is  Komuro $N$-expansive on a compact surface $M$ then  $\Omega(X)=M$.
\end{lemma}
\begin{proof}
To prove this result we first recall that in \cite[Theorem 3.4]{A4}
 it is proved that if $X$ is a flow on a compact surface $M$
 with $\Omega(X)\neq M$ then there is
 a reparameterization of the flow which is not separating.
 We wish to remark that from the proof of
 \cite[Theorem 3.4]{A4} we have:
if $l$ is an arc, transversal to the flow and contained in the wandering set, then for all $\delta>0$ there is a subarc $l'\subset l$ and
$\alpha\in\cH(l')$ such that
$\diam\mathcal{X}^t_\alpha(l')<\delta$ for all $t\in\mathbb{R}$.
As $l$ is wandering, we can assume that different points
of $l'$ are in different orbits.
Thus, taking $A\subset l'$ with $N+1$ points we conclude that
$X$ cannot be Komuro $N$-expansive.
This proves that for a Komuro $N$-expansive flow of a compact surface $M$ it holds that $\Omega(X)=M$.
\end{proof}

Now we are ready to prove our last result in this paper:

\begin{proof}[Proof of Theorem \ref{D}]
Since every Komuro expansive flow is Komuro $N$-expansive, we
suppose that $X$ is Komuro $N$-expansive on the compact surface $M$.
We will show that $X$ is Komuro expansive.
From Lemmas \ref{l-sing-finito} and \ref{l-omega-tudo}
we know that $\sing(X)$ is finite and $\Omega(X)=M$.
Therefore, if $p\in M$ is a periodic (non-singular) point then every point close to $p$ has to be periodic too. This is because the return map to a small transversal though $p$ cannot have wandering points.
Thus, we obtain an open cylinder of periodic points (\emph{i.e.,} an open arc which is transversal to the flow, whose first return map is the identity). This easily contradicts the $N$-expansivity of $X$. Thus, $X$ has no periodic orbit.

Now we apply \cite[Theorem 6.7]{A5}, where
it is proved that if $M$ is a compact surface different of the torus, a flow $X$
is Komuro expansive (there called, simply, 'expansive') if and only if $\Omega(X)=M$, $X$ has not periodic orbits and $\sing(X)$ is
finite.
Therefore, we conclude that, excluding the torus, every Komuro $N$-expansive flow on a compact surface is Komuro expansive (Komuro 1-expansive).

Finally we consider the case of the torus.
By \cite[Remark 6.6]{A5} we have that the torus does not admit Komuro expansive flows. Thus, it only remains to prove that there are no Komuro $N$-expansive flows on the torus.
We assume that $M$ is the torus.
Since $\sing(X)$ is finite, $\Omega(X)=M$ and there are no periodic orbits we can apply
\cite[Theorem 3]{CGL} to conclude that each singularity is
a \emph{multi-saddle}, see Figure \ref{figuraMultiSaddle}.
In particular, $X$ cannot have atracting or repelling singular points (\emph{i.e.}, singular points of positive index).
Applying PoincarÃ©-Hopf Theorem, the flow can only have singular points of index 0.
This kind of singularities are locally conjugate to
the flow defined by the plane vector field $Y(x,y)=(x^2+y^2,0)$ near the origin.
Each singular point of these kind can be \emph{removed} by replacing
a rectangle around it with a standard flow-box (for instance, generated by the vector field $Z(x,y)=(1,0)$).
After removing these singularities we can apply
\cite[Lemma 4.1]{A5} to conclude that
the flow we obtain is a minimal (irrational) flow.
That is, the original flow is obtained from a minimal flow
by adding a finite number of index 0 singularities.
Let $l$ be a small arc transverse to the flow.
In this arc we can take $N+1$ points whose orbits do not converge to a singularity neither for $t\to+\infty$ nor $t\to-\infty$. Therefore, their orbits are reparameterizations of $N+1$ orbits of an irrational flow.
Since irrational flows are isometries, we see that the trajectories of these $N+1$ points in the arc $l$
can be reparameterized in order to contradict the $N$-expansivity.
This proves that Komuro $N$-expansivity implies Komuro expansivity on compact surfaces.

Now we give examples of Komuro cw-expansive flows of compact surfaces which are not Komuro $N$-expansive. They are similar to those considered in the previous paragraph. Let $Y$ be a vector field generating a minimal flow of the torus.
Take a non-negative smooth function $\rho$ on the torus vanishing only at some point $p\in M$.
Let $X^t$ be the flow generated by the vector field $\rho Y$.
If we take a continuum $A$ which is not contained in an orbit, we can take two points $x,y\in A$ such that $X^t(x)\to p$ as $t\to +\infty$ and the positive orbit of $y$ is dense in $M$.
Therefore,
$\diam(\mathcal{X}^t_\alpha(A))$ is bounded away from zero, independently from $A$ and $\alpha$.
This implies that $X^t$ is Komuro cw-expansive.
As explained above, $X^t$ cannot be $N$-expansive and the proof ends. Other examples with inflinitely many singularities or wandering points can be found in \cite{A3}.
\end{proof}

%
%
%

\vspace{0.2cm}
\noindent
{\em  M.J. Pacifico},
\noindent Instituto de Matem\'atica,
Universidade Federal do Rio de Janeiro,
C.P. 68.530\\ CEP 21.945-970,
Rio de Janeiro, RJ, Brazil.\\
E-mail: pacifico@im.ufrj.br 
\vspace{0.2cm}

\noindent
{\em W. Cordeiro},
\noindent
Faculty of Mathematics and Computer Science, Nicolaus Copernicus University \\
Chopina 12/18, 87-100 Toru\'n, Poland \\
E-mail: welingtonscordeiro@gmail.com
\vspace{0.2cm}

\noindent
{\em A. Artigue},
\noindent
Departamento de MatemÃ¡tica y EstadÃ­stica del Litoral,
Centro Universitario Regional Litoral Norte, Universidad de la RepÃºblica, C.P. 50000, Rivera 1350, Salto, Uruguay.\\
 E-mail: artigue@unorte.edu.uy

\end{document}